\documentclass{article}

\usepackage{amsmath, amssymb, amsthm}

\usepackage{float,graphicx}

\newtheorem{thm}{Theorem}
\newtheorem{lem}{Lemma}

\newcommand{\calu}{\mathcal{U}}
\newcommand{\calm}{\mathcal{M}}
\newcommand{\calv}{\mathcal{V}}
\newcommand{\cale}{\mathbb{E}}

\DeclareMathOperator{\spn}{span}

\title{The vector graph and the chromatic number of the plane, or how NOT to prove that $\chi(\mathbb{E}^2)>4$}
\author{Jeremy F. Alm\\ Illinois College\\ Jacksonville, IL 62650\\ \texttt{alm.academic@gmail.com} \and Jacob Manske\\ Epic Systems Corporation\\ Verona, Wisconsin 53593\\ \texttt{jmanske@gmail.com}}

\begin{document}
\maketitle


\begin{abstract}

The chromatic number $\chi\left(\cale^2\right)$ of the plane is known to be some integer between 4 and 7, inclusive. We prove a limiting result that says, roughly, that one cannot increase the lower bound on $\chi\left(\cale^2\right)$ by pasting  Moser spindles together --- even countably many --- by taking translations by $\mathbb{Z}$-linear combinations of a certain set of vectors.

\end{abstract}

\section{Introduction and motivation}
Let $n$ be a positive integer, and let $[n] = \left\{1,2,\ldots,n\right\}$. Let $\cale^{n}$ denote the graph whose vertex set is $\mathbb{R}^{n}$, where two vertices are adjacent if and only if they are distance 1 apart using the standard Euclidean metric.
For a positive integer $k$, a $k$-coloring of a graph $G$ is a function $f: V(G) \to [k]$. A $k$-coloring is \emph{proper} if whenever $x$ and $y$ are adjacent in $G$, $f(x) \neq f(y)$.
The \emph{chromatic number} of a graph $G$ (which we shall denote by $\chi\left(G\right)$) is the smallest positive integer $k$ such that there exists a proper $k$-coloring of $G$.

It is known that
\begin{equation}\label{basicbounds}
4 \leq \chi\left(\cale^2\right) \leq 7.
\end{equation}
These bounds can be found summarized in the text by Soifer~\cite{soiferbook}. The lower bound uses the famous 7-vertex unit distance graph known as the \emph{Moser Spindle}, which we will discuss below. The correct value of $\chi\left(\cale^2\right)$ may depend on the axioms of set theory, as shown by Shelah and Soifer in~\cite{shelahsoifer}. In particular, they give a graph $G$ such that
\begin{enumerate}
\item[(i.)] in ZFC, $\chi(G)=2$;
\item[(ii.)] in ZF + countable choice + LM, $\chi(G)\geq \aleph_1$.
\end{enumerate}
ZF denotes the axioms of Zermelo-Fraenkel set theory, ZFC denotes the axioms of Zermelo-Frankel set theory together with the axiom of choice, and  LM denotes the axiom that says that all subsets of the plane are Lebesgue measurable.

Also, $\chi\left(\cale^{3}\right)$ has been widely studied. A lower bound for $\chi\left(\cale^{3}\right)$ was achieved by Ra{\u\i}ski{\u\i} in~\cite{raiskii}, when he showed that $\chi\left(\cale^{n}\right) \geq n+2$, so $\chi\left(\cale^{3}\right)\geq 5$. Currently, the best lower bound for $\chi \left(\cale^{3}\right)$ is due to Nechushtan in~\cite{nechushtan}, who showed that $\chi\left(\cale^{3}\right) \geq 6$. An upper bound of $\chi\left(\cale^{3}\right) \leq 18$ was achieved by   Coulson in~\cite{coulson18}; he later improved this result to $\chi\left(\cale^{3}\right) \leq 15$  in~\cite{coulson15}. In summary, the best bounds for $\chi\left(\cale^{3}\right)$ stand at
\[ 6 \leq \chi\left(\cale^{3}\right) \leq 15.\]

\section{The vector graph and main results}


We introduce the \emph{vector graph} $\calv_{\calu}$, a subgraph of $\cale^2$, that depends on a set of unit vectors $\calu$ from $\mathbb{R}^{2}$.  Given a set $\calu$ of unit vectors, we define \[\spn \calu =\left\{a_{1}u_{1} +\cdots + a_{n}u_{n} : n \in \mathbb{N}, a_{i} \in \mathbb{Z}, u_{i}\in \calu\right\}.\]
The vertices of $\calv_{\calu}$ are the vectors in $\spn \calu$ interpreted as ordered pairs of real numbers.
Adjacency in $\calv_{\calu}$ is given by $u \sim v$ if and only if $u - v \in \calu$ or $v - u \in \calu$. For example, if we take $\calu = \left\{(1,0),(0,1)\right\}$, then $\calv_{\calu}$ is the (Euclidean) unit distance graph whose vertex set is $\mathbb{Z}^{2}$.

For a vector $v$ from $\mathbb{R}^{2}$, let $r_{\theta}\left(v\right)$ denote the counterclockwise rotation of the vector $v$ by $\theta$ radians. Set $\alpha = \arccos\left(\dfrac{1}{2\sqrt{3}}\right) - \dfrac{\pi}{6}$ and $\beta = \alpha + \dfrac{\pi}{3}$. Let $u$ be any unit vector from $\mathbb{R}^{2}$ and define a set of vectors $\calm$ as
\begin{equation}\label{moservectorset}
\calm = \left\{u, r_{\alpha}(u), r_{-\alpha}(u), r_{\beta}(u), r_{-\beta}(u), r_{\alpha}(u) - r_{\beta}(u), r_{-\alpha}(u)-r_{-\beta}(u)\right\}.
\end{equation}
The reason for choosing these particular angles $\alpha$ and $\beta$ is so that $\calv_{\calm}$ contains a copy of the Moser spindle. See Figure \ref{moservectorfig} for a depiction of a subgraph of $\calv_{\calm}$ and Figure \ref{fig:section} for seven copies of the spindle pasted together, which begins to illustrate the complexity of $\calv_{\calm}$.

\begin{figure}[H]
\begin{center}
\includegraphics[width=2in]{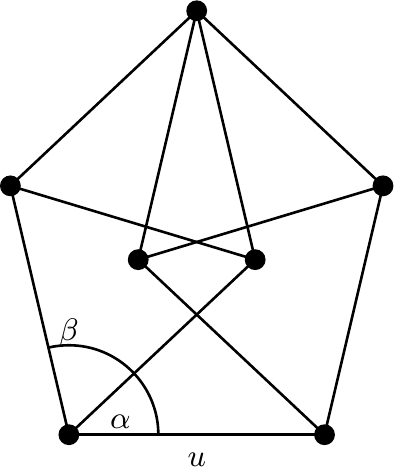}
\end{center}
\caption{A subgraph of $\calv_{\calm}$ indicating the angles $\alpha$ and $\beta$ and a unit vector $u$.}
\label{moservectorfig}
\end{figure}

\begin{figure}[H]
\begin{center}
\includegraphics[width=2in]{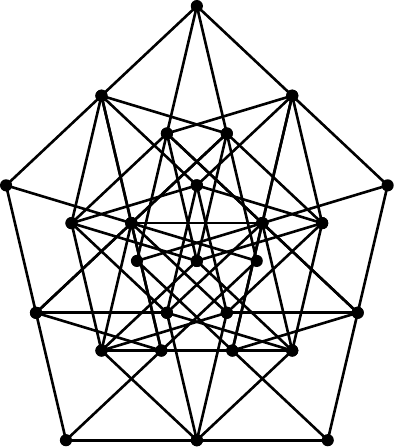}
\end{center}
\caption{A subgraph of $\calv_\calm$.}
\label{fig:section}
\end{figure}

\begin{lem}\label{thm:moser}
For all $v_1\in\calv_{\calm}$ and for all vertices $w$ in the  Moser Spindle, there exist $v_2,\ldots,v_7\in \calv_\calm$ such that the subgraph of $\calv_\calm$ induced by $\{v_1,\ldots,v_7\}$ is isomorphic to the Spindle via an isomorphism that carries $v_1$ onto $w$.
\end{lem}

\begin{proof}
Let $v_1$ be a vector in  $\calv_{\calm}$. We are
done if we can show that there are three Moser Spindles in  $\calv_{\calm}$ such that $v_1$ occurs
as the bottom left vertex in one spindle, the left most vertex in a second spindle,
and the top most vertex in a third spindle (see Figure \ref{fig:labels}). To see that there are
three such spindles, add $v_1$ to all coordinates of the spindle in Figure \ref{fig:labels}. This gives a copy of the
Moser spindle with $v_1$ as the bottom left vertex. Similarly, adding $v_1 - r_\beta(u)$ and
$v_1 - (r_\alpha(u) + r_\beta(u))$ gives spindles with $v_1$ as the left most vertex and the top most
vertex, respectively.

\end{proof}

\begin{figure}[H]
\begin{center}
\includegraphics[width=3in]{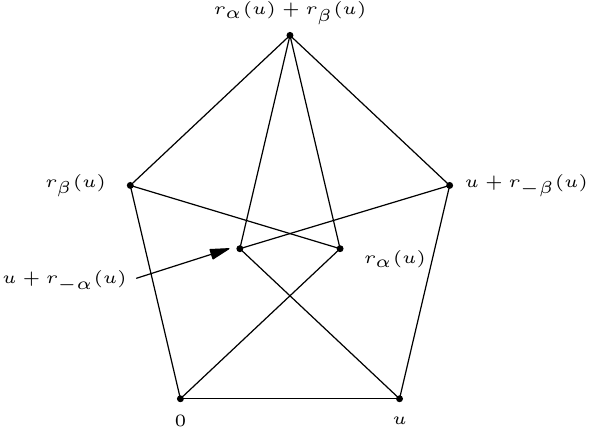}
\end{center}
\caption{A subgraph of $\calv_\calm$ with vertices labelled with vectors}
\label{fig:labels}
\end{figure}


Since the chromatic number of the Moser spindle is 4, we have $\chi\left(\calv_{\calm}\right) \geq 4$. One might think that the vector graph $\calv_\calm$ might not be 4-colorable, in the same way that the Spindle, which is made up of triangles pasted together, is not 3-colorable. In fact, in Theorem \ref{thm:moservector}, we show that $\calv_{\calm}$ is 4-colorable.
\begin{thm}\label{thm:moservector} For any unit vector $u$ in $\mathbb{R}^{2}$, if $\calv_\calm$ is the vector graph given by (\ref{moservectorset}), then 
$\chi\left(\calv_{\calm}\right) = 4$.
\end{thm}


\begin{proof}[Proof of Theorem \ref{thm:moservector}] Let $u$ be any unit vector in $\mathbb{R}^{2}$, and let $\calm$ be the set described in (\ref{moservectorset}). Since the chromatic number of the Moser Spindle is $4$, we need only to show that $\chi\left(\calv_{\calm}\right) \leq 4$.

For ease in notation, and when there is no confusion, we let $\alpha = r_{\alpha}(u)$, $\overline{\alpha} = r_{-\alpha}(u)$, $\beta = r_{\beta}(u)$, and $\overline{\beta} = r_{-\beta}(u)$. Hence, we may write
\[\calm = \left\{u, \alpha, \overline{\alpha}, \beta, \overline{\beta}, \alpha - \beta, \overline{\alpha} - \overline{\beta}\right\}.\]
Notice that the set of vectors $\left\{\alpha, \overline{\alpha}, \beta, \overline{\beta}\right\}$ is linearly independent over $\mathbb{Z}$. As such, for all $v \in \calv_{\calm}$ there is a unique $\mathbb{Z}$-linear combination such that $v= a\alpha +b\overline{\alpha}+ c\beta + d\overline{\beta}$.  Define $f: \calv_{\calm} \to \mathbb{Z}_{4}$ by 
\[
    f(v) = a + 3b + 2c + d \pmod{4}.
\]


We show now $f$ is a proper 4-coloring. Let $x$ and $y$ be vertices from $\calv_{\calm}$ with $x \sim y$. Suppose $f(x) = f(y)$. Hence $f(x - y) = 0$, as $f$ is linear. Since $x \sim y$, $x - y \in \calu$, so we need only check that no vector in $\calm$ gets mapped to 0.

Since $u = \alpha + \beta + \overline{\alpha} + \overline{\beta}$, $f(u) = 1 + 2 + 3 + 1 \equiv 3\ \text{(mod 4)}$. We also have $f(\alpha - \beta)  \equiv  1 - 2  \equiv  3\ \text{(mod 4)}$ and $f\left(\overline{\alpha} - \overline{\beta}\right)  \equiv  3 - 1  \equiv  2\ \text{(mod 4)}$, so $f$ is a proper 4-coloring of $\calv_{\calm}$ and the proof is complete.
\end{proof}



\section{Future work/discussion}

It could be that proper 4-colorings of the plane (if they exist) occupy a place analogous to that of non-principal ultrafilters\footnote{The existence of non-principal ultrafilters is proved as follows: apply Zorn's Lemma to the poset (under containment) of all filters over an infinite set that contain all of the co-finite subsets; the resulting maximal element in this poset is an ultrafilter that is necessarily non-principal.
}; we could never write one down.  For instance, one could imagine a proof that all finite unit-distance graphs are 4-colorable; this would imply (by Erd\H{o}s-DeBruijn \cite{dBE}) that a 4-coloring of the plane exists, but it certainly would not give us the means to write one down.

It could be that the presence or absence of the axiom of choice plays a  central role. Shelah and Soifer \cite{shelahsoifer}, relying on a result of Falconer \cite{Falconer81}, prove that if all finite unit distance graphs are 4-colorable, then even though  $\chi(\mathbb{E}^2)=4$ in ZFC, in ZF + dependent choice + LM we have that $\chi(\mathbb{E}^2)\geq 5$.

\section{Acknowledgements}
We thank the referees for helpful comments toward improving the exposition of the paper.



\end{document}